\documentclass[12pt]{amsart}
\usepackage{amsfonts,amssymb,latexsym,amsmath, amsxtra,verbatim}
\usepackage{amssymb}
\usepackage{amsmath, amsthm}
\usepackage{amsthm,amssymb,amstext,amscd,amsfonts,amsbsy,amsxtra,latexsym,amsmath,mathtools}
\usepackage{multirow,hyperref}
\hypersetup{colorlinks=true,linkcolor=blue,citecolor=red}
\usepackage[all]{xy}
\usepackage{graphicx}
\usepackage{enumitem}
\usepackage{color}
\theoremstyle{plain}
\newtheorem{theorem}{Theorem}[section]

\newtheorem{lemma}[theorem]{Lemma}
\newtheorem{proposition}[theorem]{Proposition}
\newtheorem*{conjecture*}{Conjecture}
\newtheorem*{challenge*}{Open Problem}

\theoremstyle{definition}
\newtheorem{definition}[theorem]{Definition}
\theoremstyle{remark}
\newtheorem*{remark}{Remark}
\newtheorem*{remarks}{Remarks}

\numberwithin{equation}{section}
\newtheorem*{rmk*}{Remark}

\newcommand{\R}{\mathbb R}
\newcommand{\N}{\mathbb N}
\newcommand{\colvector}[2]{\left(\begin{smallmatrix} #1 \\ #2 \end{smallmatrix}\right)}
\newcommand{\Z}{\mathbb Z}
\newcommand{\C}{\mathbb C}

\newcommand{\Q}{{\mathbb Q}}

\makeatletter
\newcommand{\vast}{\bBigg@{4}}
\makeatother

\newcommand{\bea}{\begin{eqnarray}}
\newcommand{\eea}{\end{eqnarray}}
\newcommand{\be}{\begin{equation}}
\newcommand{\ee}{\end{equation}}
\newcommand{\sign}{\operatorname{sgn}}
\newcommand{\benn}{\begin{equation*}}
\newcommand{\eenn}{\end{equation*}}

\def\GL{\rm GL}

\def\({\left(}
\def\){\right)}

\def\SL{{\rm SL}}

\def\k2{\frac{k}{2}}

\begin{document}
\thanks{The research of the first author is supported by the Alfried Krupp Prize for Young University Teachers of the Krupp foundation and the research leading to these results receives funding from the European Research Council under the European Union's Seventh Framework Programme (FP/2007-2013) / ERC Grant agreement n. 335220 - AQSER. The third author thanks the University of Cologne and the DFG for their generous support via the University of Cologne postdoc grant DFG Grant D-72133-G-403-151001011, funded under the Institutional Strategy of the University of Cologne within the German Excellence Initiative. }
\date{\today}

\title[A special family of Maass forms]{On some special families of $q$-hypergeometric Maass forms}
\author{Kathrin Bringmann}
\address{Mathematical Institute, University of Cologne, Weyertal 86-90, 50931 Cologne, Germany}
\email{kbringma@math.uni-koeln.de}
\author{Jeremy Lovejoy}
\address{CNRS LIAFA	Universite Denis Diderot - Paris 7, Case 7014, 75205 Paris Cedex 13, France}
\email{lovejoy@math.cnrs.fr}
\author{Larry Rolen}
\address{Hamilton Mathematics Institute \& School of Mathematics, Trinity College, Dublin 2, Ireland}
\email{lrolen@maths.tcd.ie}
\begin{abstract}
Using special polynomials 
related to the Andrews-Gordon identities and the colored Jones polynomial of torus knots, we construct classes of $q$-hypergeometric series lying in the Habiro ring. These give rise to new families of quantum modular forms, and their Fourier coefficients encode distinguished Maass cusp forms. The cuspidality of these Maass waveforms is proven by making use of the Habiro ring representations of the associated quantum modular forms. Thus, we provide an example of how the $q$-hypergeometric structure of the associated series to can be used to establish modularity properties which are otherwise non-obvious. We conclude the paper with a number of motivating questions and possible connections with Hecke characters, combinatorics, and still mysterious relations between $q$-hypergeometric series and the passage from positive to negative coefficients of Maass waveforms.
\end{abstract}
\maketitle
\section{Introduction and Statement of Results}\label{Intro}
We begin by introducing the special polynomials $H_n(k,\ell;b;q)$ which play a key role in our constructions. To do so, we recall the \emph{$q$-rising factorial}, defined by
\begin{equation*}
(a)_n = (a;q)_n := \prod_{k=0}^{n-1} \big(1-aq^{k}\big),
\end{equation*}
along with the \emph{Gaussian polynomials}, given by
\begin{equation*} 
\begin{bmatrix} n \\ k \end{bmatrix}_q  :=
\begin{cases}
\frac{(q)_n}{(q)_{n-k}(q)_k} & \text{if $0 \leq k \leq n$}, \\
0 & \text{otherwise}. \notag
\end{cases}
\end{equation*}
Then for $k\in\mathbb{N}$, $1 \leq \ell \leq k$, and $b\in\{0,1\}$, we define the polynomials $H_{n}(k,\ell;b;q)$ by
\begin{equation}\label{Hdef}
H_{n}(k,\ell;b;q) := \sum_{n = n_k \geq n_{k-1} \geq \ldots \geq n_1 \geq 0} \prod_{j=1}^{k-1} q^{n_j^2+(1-b)n_j} \begin{bmatrix} n_{j+1}-n_j - bj + \sum_{r=1}^j (2n_r + \chi_{\ell > r}) \\ n_{j+1}-n_j \end{bmatrix}_q.
\end{equation}
Here we use the usual charactersitic function $\chi_{A}$, defined to be $1$ if $A$ is true and $0$ otherwise.

These polynomials occurred explicitly (in the case $b=1$) in recent work on torus knots \cite{Hi-Lo1}, and they can also be related to generating functions for the partitions occurring in Gordon's generalization of the Rogers-Ramanujan identities \cite{Wa1}.  To describe the latter, let $G_{k,i,i',L}(q)$ be the generating function for partitions of the form
\begin{equation} \label{A-Grelation}
\sum_{j=1}^{L-1} j f_j,
\end{equation}
with $f_1 \leq i-1$, $f_{L-1} \leq i'-1$, and $f_i + f_{i+1} \leq k$ for $1 \leq k \leq L-2$.   Using the fact that
\begin{equation*}
\begin{bmatrix} n \\ k \end{bmatrix}_{q^{-1}} = q^{-k(n-k)}\begin{bmatrix} n \\ k \end{bmatrix}_{q},
\end{equation*}
making some judicious changes of variable and comparing with Theorem 5 of \cite{Wa1}, it can be shown that
\begin{equation} \label{A-Grelationbis}
H_n\big(k,\ell;b,q^{-1}\big) = q^{(k-1)bn - 2(k-1)\binom{n+1}{2}} G_{k-1,\ell,k,2n-b+1}(q).
\end{equation}

In the context of torus knots, the $n$-th coefficient in Habiro's cyclotomic expansion of the colored Jones polynomial of the left-handed torus knot $T(2,2k+1)$ was shown in \cite{Hi-Lo1} to be $q^{n+1-k}H_{n+1}(k,1;1;q)$, and the general $H_{n}(k,\ell;1;q)$ were used to construct a class of $q$-hypergeometric series with interesting behavior both at roots of unity and inside the unit circle. As we shall see shortly, this is the heart of the quantum modular phenomenon; the reader is also referred to \cite{Hi-Lo1} for more details.

In this paper, we consider classes of $q$-hypergeometric Maass cusp forms constructed from the polynomials $H_{n}(k,\ell;b;q)$.  These functions, denoted {by} $F_j(k,\ell;q)$  $\big(j\in\{1,2,3,4\}\big)$, are defined as follows:\footnote{Note that $F_2(k,\ell;q)$ has a convergence issue, which we overcome by averaging over the even and odd partial sums with respect to $n$.}
\begin{equation}\label{FFnsDefn}\begin{aligned}
F_1(k,\ell;q)  &:= \sum_{n \geq 0} (q)_{n}(-1)^{n}q^{\binom{n+1}{2}} {H}_{n}(k,\ell;0;q), \\ 
F_2(k,\ell;q)  &:= \sum_{n \geq 0} \big(q^2;q^2\big)_{n}(-1)^{n} {H}_{n}(k,\ell;0;q), \\
F_3(k,\ell;q)  &:= \sum_{\substack{n \geq 1}} (q)_{n-1}(-1)^{n}q^{\binom{n+1}{2}} H_{n}(k,\ell;1;q),  \\ 
F_4(k,\ell;q) &:= \sum_{\substack{n \geq 1}} (-1)_{n}(q)_{n-1}(-q)^{n} H_{n}(k,\ell;1;q).  
\end{aligned}
\end{equation}
Note that when $k=1$ the polynomials in \eqref{Hdef} are identically $1$, and so the above contain two celebrated $q$-series of Andrews, Dyson, and Hickerson \cite{An-Dy-Hi1} as special cases.
\noindent Namely, we have
 \begin{equation}\label{ConnectionOurFamilySigma}
 2F_2(1,1;q) = \sigma\big(q^2\big)
 \end{equation}
 and
 \begin{equation}\label{ConnectionOurFamilySigmaStar}
 F_4(1,1;q) = -\sigma^*(-q),
 \end{equation}
where
\begin{align}\label{sigmadef}
\sigma(q) :=& \sum_{n \geq 0} \frac{q^{\binom{n+1}{2}}}{(-q)_n} \\
		   =&\  1 + \sum_{n \geq 0} (-1)^nq^{n+1}(q)_n \label{id1}\\
		   =&\   2\sum_{n \geq 0}(-1)^n(q)_n , \label{id2}\\
\sigma^*(q) :=&\  2\sum_{n \geq 1} \frac{(-1)^nq^{n^2}}{(q;q^2)_n}  \label{sigma*def}\\
	=&\  -2\sum_{n \geq 0} q^{n+1}\big(q^2;q^2\big)_n. \label{id3}
\end{align}
The definitions in \eqref{sigmadef} and \eqref{sigma*def} are the original definitions of Andrews, Dyson, and Hickerson, while the identities \eqref{id1} and \eqref{id3} were established by Cohen \cite{Co1}, and \eqref{id2} follows easily.
 The function $\sigma$ was first considered in Ramanujan's
``Lost'' notebook (see \cite{AndrewsLostNotebookV}). Andrews, Dyson, and Hickerson showed \cite{An-Dy-Hi1} that this series satisfies several striking and beautiful properties, and in particular that if
$
\sigma(q)=\sum_{n\geq0}S(n)q^n
,$
then $\lim \sup |S(n)|=\infty$ but $S(n)=0$ for infinitely many $n$. Their proof is closely related to indefinite theta series representations of $\sigma$, such as:
\[
\sigma(q)=\sum_{\substack{n\geq0\\ |\nu|\leq n}}(-1)^{n+\nu}q^{\frac{n(3n+1)}2-\nu^2}\big(1-q^{2n+1}\big)
.
\]
The coefficients of $\sigma^*(q)$ have the same properties.

Subsequently Cohen \cite{Co1} showed how to nicely package the $q$-series of Andrews, Dyson, and Hickerson within a single modular object.
Namely, he proved that if coefficients $\{T(n)\}_{n\in1+24\Z}$ are defined by
\begin{equation} \label{Tofndef}
\sigma\big(q^{24}\big)=\sum_{n\geq0}T(n)q^{n-1}
,
\quad\quad\quad\quad
\sigma^*\big(q^{24}\big)=\sum_{n<0}T(n)q^{1-n},
\end{equation}
then the $T(n)$ are the Fourier coefficients of a Maass waveform. The definitions of Maass waveforms and the details of this construction are reviewed in Section \ref{MaassFormsSctn}. In this paper, we show that the functions $F_j(k,\ell;q)$ have a similar connection to Maass waveforms, and by \eqref{ConnectionOurFamilySigma} and \eqref{ConnectionOurFamilySigmaStar} may thus be considered as a $q$-hypergeometric framework containing the examples of Andrews, Dyson, and Hickerson and Cohen.\\
\indent
In what follows, we let $f$ be a Maass waveform with eigenvalue $1/4$ (under the hyperbolic Laplacian $\Delta$)  on a congruence subgroup of $\SL_2(\mathbb{Z})$ (and with a possible multiplier), which is cuspidal at $i\infty$. If the
Fourier expansion of $f$ is given, as in Lemma \ref{Maass0Fourier}, by $(\tau=u+iv)$
\[
f(\tau)
=
v^{\frac{1}{2}}\sum_{n\neq0}A(n)K_{0}\bigg(\frac{2\pi |n|v}{N}\bigg)e\bigg(\frac{n u}{N}\bigg)
,
\]
where
$e(w):=e^{2\pi i w}$, then the $q$-series associated to the positive coefficients of $f$ is defined by
\begin{equation}\label{pluspart}
f^+(\tau):=\sum_{n>0}A(n)q^{\frac{n}{N}}.
\end{equation}
\indent
We remark in passing that such a map from Maass forms to $q$-series was studied extensively by
Lewis and Zagier \cite{LewisZagier1,LewisZagier2}, and, as we shall see, was used by Zagier \cite{Za1} to show that
such functions are quantum modular forms. Such a construction is also closely related to the study of automorphic distributions in \cite{MS}.
\begin{theorem}\label{mainthm1}
For any $k, \ell\in\mathbb{N}$, with $1 \leq \ell \leq k$, and $j\in\{1,2,3,4\}$, there exists a Maass cusp form $G_{j,k,\ell}$ with eigenvalue $1/4$  for some congruence subgroup of $\operatorname{SL}_2(\Z)$, such that
\[
G_{j,k,\ell}^+(\tau)
=
q^{\alpha}
F_j\big(k,\ell;q^d\big)
\]
for some $\alpha\in\Q$, where $d=1$ if $j\in\{1,3\}$ and $d=2$ if $j\in\{2,4\}$ .
\end{theorem}
\begin{remark}
The cuspidality of the Maass waveform $G_{j,k,\ell}$ is far from obvious. Indeed, the authors are aware of only two approaches to prove such a result: either to explicitly write down representations for the $G_{j,k,\ell}$ in terms of Hecke characters (which we suspect exist, but which we were unable to identify), or, as we show below, to use the $q$-hypergeometric representations of the $F_j$ directly. This connection, which was hinted at for certain examples in \cite{RobMaass}, utilizes $q$-hypergeometric series to deduce modularity properties in an essential way.
\end{remark}

The proof of Theorem \ref{mainthm1} relies on the Bailey pair machinery and important results of Zwegers \cite{ZwegersMockMaass} giving modular completions for indefinite theta functions of a general shape. In particular, this family of indefinite theta functions
naturally describes the behavior of functions studied by many others in the literature, as described in a recent proof of Krauel, Woodbury, and the second author \cite{KRW} of unifying conjectures of Li, Ngo, and Rhoades \cite{RobMaass}. The indefinite theta functions considered here are given for $M\in\N_{\ge2}$ and vectors $a=(a_1,a_2)\in\Q^2$ and $b=(b_1,b_2)\in\Q^2$ such that $a_1 \pm a_2 \not \in \mathbb{Z}$:
\begin{equation} \label{Sdef}
\begin{aligned}
&
S_{a,b;M}(\tau)
:=
\\
&
\Bigg(\sum_{n\pm \nu\geq-\lfloor a_1\pm a_2\rfloor}+\sum_{n\pm\nu<-\lfloor a_1\pm a_2\rfloor}\Bigg)e\big((M+1)b_1n-(M-1)b_2\nu\big)q^{\frac12\big((M+1)(n+a_1)^2-(M-1)(\nu+a_2)^2\big)}.
\end{aligned}
\end{equation}
\indent
The next theorem states conditions under which $S_{a,b;M}$ is the image of a Maass waveform under the map defined in \eqref{pluspart}.   The definitions of $\gamma_M$, the equivalence relation $\sim$, and the operation $^*$ are given in Section \ref{ZwegersWorkSection}.
\begin{theorem}\label{mainthm2}
Suppose that $a,b\in\Q^2$ with $a\neq0$, $a_1\pm a_2\not\in\Z$, $M\in\N_{\geq2}$, and $(\gamma_M a,\gamma_M b)\sim(a,b)$ or both  $(\gamma_M a,\gamma_M b)\sim(a^*,b^*)$ and $(\gamma_M a^*,\gamma_M b^*)\sim(a,b)$ hold. Then $S_{a,b;M}=F^+$ for a Maass waveform $F$ of eigenvalue $1/4$ on a congruence subgroup of $\operatorname{SL}_2(\Z)$.
\end{theorem}
\begin{remark}
The family of $q$-series $F_j$ in Theorem \ref{mainthm1} are all specializations of the series $S_{a,b;M}$. We note that although all of the functions $F_j$ are cusp forms, for a general Maass form $F$ in Theorem \ref{mainthm2} this is not always true. For example, the function $W_1$ considered in Theorem 2.1 of \cite{RobMaass} is shown not to be a cusp form, and using Theorem 4.2 and (4.4) one can easily check that the function $W_1$ fits into the family $S_{a,b;M}$.
\end{remark}

In addition to the relations of the $q$-series $F_j$ to Maass forms, following Zagier's work, we find that these functions are instances of so-called quantum modular forms \cite{Za1}. These new types of modular objects, which are reviewed in Section \ref{MaassFormsHolomorphization}, are connected to many important combinatorial generating functions, knot and $3$-manifold invariants, and are intimately tied to the important volume conjecture for hyperbolic knots. Roughly speaking, a \emph{quantum modular form} is a function which is defined on a subset of $\mathbb Q$ and whose failure to transform 
modularly is described by  a particularly ``nice'' function.   (See Definition \ref{cocycle}.)    Viewed from a general modularity framework, the generating function of the set of positive coefficients of a Maass form automatically has quantum modular transformations when considered as (possibly divergent) asymptotic expansions. The situation becomes much nicer when, as happens for $\sigma$ and $\sigma^*$, $q$-hypergeometric representations can be furnished which show convergence at various roots of unity (as the series specialize to finite sums of roots of unity). This quantum modularity result, as well as the relation of the associated quantum modular forms to the cuspidality of the Maass form is described in Theorem \ref{MaassQMFThm}. In particular, if such a $q$-series is an element of the Habiro ring, which essentially means that it can be written as
\[
\sum_{n\geq0}a_n(q)(q)_n
\]
for polynomials $a_n(q)\in\Z[q]$, then it is apparent that it converges at all roots of unity $q$, and hence the associated Maass form is cuspidal. This observation, combined with Zagier's ideas, yields the following corollary (the definitions of quantum modular forms and related terms are given in Section \ref{MaassFormsHolomorphization}).
\begin{theorem}\label{mainthm3}
For any choice of $j,k,\ell$ as in Theorem \ref{mainthm1}, the functions $F_{j,k,\ell}$ are quantum modular forms of weight $1$ on a congruence subgroup with quantum set $\mathbb P^1(\Q)$. Moreover, the cocycles $r_{\gamma}$, defined in \eqref{cocycle}, are real-analytic on $\R\setminus\{\gamma^{-1}i\infty\}$.
\end{theorem}

The paper is organized as follows. In Section \ref{PrelimSection}, we recall the basic preliminaries and definitions needed for the proofs and explicit formulations of the main theorems, which are then proven in Section \ref{ProofsSection}. As mentioned above, the main tools are the Bailey pair method, work of Zwegers in \cite{ZwegersMockMaass}, and ideas from Zagier's seminal paper on quantum modular forms \cite{Za1}. We conclude in Section \ref{QuestionsSection} with further commentary on related questions and possible future work.

\section{Preliminaries}\label{PrelimSection}

\subsection{Bailey pairs}\label{BaileyPairsSection}
In this subsection, we briefly recall the Bailey pair machinery, which is a powerful tool for connecting $q$-hypergeometric series with series such as indefinite theta functions. The basic input of this method is a \emph{Bailey pair} relative to $a$, which is a pair of sequences $(\alpha_n,\beta_n)_{n \geq 0}$ satisfying
\begin{equation} 
\beta_n = \sum_{k=0}^n \frac{\alpha_k}{(q)_{n-k}(aq)_{n+k}}. \notag
\end{equation}
Bailey's lemma then provides a framework for proving many $q$-series identities. For our purposes, we need only a limiting form, which says that if $(\alpha_n,\beta_n)$ is a Bailey pair relative to $a$, then, provided both sums converge, we have the identity

\begin{equation} \label{limitBailey}
\sum_{n \geq 0} (\rho_1)_n(\rho_2)_n \bigg(\frac{aq}{\rho_1 \rho_2}\bigg)^n \beta_n = \frac{\Big(\frac{aq}{\rho_1}\Big)_{\infty}\Big(\frac{aq}{\rho_2}\Big)_{\infty}}{(aq)_{\infty}\Big(\frac{aq}{\rho_1 \rho_2}\Big)_{\infty}} \sum_{n \geq 0} \frac{(\rho_1)_n(\rho_2)_n\Big(\frac{aq}{\rho_1 \rho_2}\Big)^n }{\Big(\frac{aq}{\rho_1}\Big)_n\Big(\frac{aq}{\rho_2}\Big)_n}\alpha_n.
\end{equation}
For more on Bailey pairs and Bailey's lemma, see \cite{An1,An2,war}.

We record four special cases of \eqref{limitBailey} for later use.
\begin{lemma} \label{Baileylemmaspecial}
The following are identities are true, provided that both sides converge.  If $(\alpha_n,\beta_n)$ is a Bailey pair relative to $1$, then
\begin{align}
\sum_{n \geq 1} (-1)^n(q)_{n-1}q^{\binom{n+1}{2}}\beta_n &= \sum_{n \geq 1} \frac{(-1)^nq^{\binom{n+1}{2}}}{1-q^n}\alpha_n, \label{Baileya=1eq1} \\
\sum_{n \geq 1} \big(q^2;q^2\big)_{n-1} (-q)^n\beta_n &= \sum_{n \geq 1} \frac{(-q)^n}{1-q^{2n}}\alpha_n, \label{Baileya=1eq2}
\end{align}
and if $(\alpha_n,\beta_n)$ is a Bailey pair relative to $q$, then
\begin{align}
\sum_{n \geq 0} (-1)^n(q)_{n}q^{\binom{n+1}{2}}\beta_n &= (1-q)\sum_{n \geq 0} (-1)^nq^{\binom{n+1}{2}}\alpha_n, \label{Baileya=qeq1} \\
\sum_{n \geq 0} \big(q^2;q^2\big)_{n} (-1)^n\beta_n &= \frac{1-q}{2}\sum_{n \geq 0} (-1)^n\alpha_n. \label{Baileya=qeq2}
\end{align}
\end{lemma}
\begin{proof}
For the first two we set $a=1$ in \eqref{limitBailey}, take the derivative $\frac{d}{d\rho_1} \big | _{\rho_1=1}$, and let $\rho_2 \to \infty$ or $\rho_2 = -1$.   For the second two we set $a=q$, $\rho_1=q$, and let $\rho_2 \to \infty$ and $\rho_2 = -q$, respectively.
\end{proof}

\subsection{Maass waveforms and Cohen's example}\label{MaassFormsSctn}
We now recall the basic definitions and facts from the theory of Maass waveforms. The interested reader is also referred to \cite{Bump, Iwaniec02} for more details.
Maass waveforms, or simply Maass forms, are functions on $\mathbb H$ which transform like modular functions but instead of being meromorphic are eigenfunctions of the hyperbolic Laplacian.
For $\tau = u+i v \in \mathbb H$ (with $u, v\in \R$), this operator is defined by
\[\Delta := -v^2 \bigg(\frac{\partial^2}{\partial u^2} + \frac{\partial^2}{\partial v^2}\bigg).\]
We also require the fact that any translation invariant function $f$, namely a function satisfying $f(\tau+1)=f(\tau)$, has
 a Fourier expansion at infinity of the form
 \begin{align*}\label{Fexpgeneral} f(\tau) = \sum_{n\in\mathbb Z} a_f(v;n) e(nu),
 \end{align*}
 where $$a_f(v;n) := \int_{0}^1 f(t+iv)e(-nt)dt.$$
Similarly, such an $f$ has Fourier expansions at any cusp $\mathfrak a$ of a congruence subgroup $\Gamma \subseteq \SL_2(\mathbb{Z})$.   We denote these Fourier coefficients by $a_{f,\mathfrak a}(v;n)$.

\begin{definition}\label{MaassForm0Def}  Let $\Gamma \subseteq  \textnormal{SL}_2(\mathbb Z)$ be a congruence subgroup.
A {\it{Maass waveform}} $f$ on $\Gamma$
with eigenvalue $\lambda=s(1-s) \in \mathbb C$ is a smooth function $f:\mathbb H \to \mathbb C$ satisfying
\begin{enumerate}[leftmargin=*,align=left]
\item[(i)] $f(\gamma \tau) = f(\tau)$ for all $\gamma \in \Gamma;$
\item[(ii)]  $f$ grows at most polynomially at the cusps;
\item[(iii)] $\Delta (f)=\lambda f$.
 \end{enumerate}
 If, moreover, $a_{f,\mathfrak a}(0;0) = 0$ for each cusp $\mathfrak a$ of $\Gamma$, then $f$ is a {\it{Maass cusp form}}.
\end{definition}
We also require the general shape of Fourier expansions of such Maass forms. As all of our forms are cusp forms, the following is sufficient for our purposes. The proof may be found in any standard text on Maass forms (such as those listed above), but we note that it follows from the differential equation (iii), growth condition (ii), and the periodicity of $f$. 
\begin{lemma} \label{Maass0Fourier}
 Let {\bf $f$} be a Maass cusp form with eigenvalue $\lambda=s(1-s)$.
  Then there exist $\kappa_1, \kappa_2, a_f(n) \in \mathbb C$, $n\neq 0$, such that
$$f(\tau) = \kappa_1 v^s + \kappa_2 v^{1-s}\delta_s(v) + v^{\frac12} \sum_{n\neq 0} a_f(n) K_{s-\frac12} (2\pi |n| v)e(nu),$$
where $K_\nu$ is the modified Bessel function of the second kind, and $\delta_s(v)$  is equal to   $\log(v)$ or $1$, depending on whether $s=1/2$ or $s\neq 1/2$, respectively. Such an expansion
also exists at all cusps.
\end{lemma}

  Cohen proved (in the notation of \eqref{Tofndef}) that the function
\begin{equation}\notag 
f(\tau)
:=
v^{\frac{1}{2}}\sum_{n\in1+24\Z}T(n)K_0\bigg(\frac{2\pi |n|{v}}{24}\bigg)e\bigg(\frac {nu}{24}\bigg)
\end{equation}
is a Maass form on the congruence subgroup $\Gamma_0(2)$ with a multiplier. Namely, $u$ satisfies the transformations
\begin{equation*}
f\bigg(\!-\frac1{2\tau}\bigg)=\overline{f(\tau)}, \qquad
f(\tau+1)=e\bigg(\frac1{24}\bigg)f(\tau),
\end{equation*}
and is an eigenfunction  of $\Delta$
with eigenvalue $1/4$.
Put another way, Cohen showed that
\[f^+(\tau)=\sigma(q),\]
(in the notation of \eqref{pluspart})
and that $\sigma^*$ similarly interprets the negative Fourier coefficients of $u$. Cohen's proof relies on connections between $\sigma,\sigma^*$ and the arithmetic of a quadratic field, which also forms the basis of investigations by many authors of the series discussed by Li, Ngo, and Rhoades in \cite{RobMaass}. However, as noted above, computing the Hecke characters related to such $q$-series using Cohen's methods quickly becomes computationally difficult. Instead, we use work of Zwegers which provides a convenient framework for giving examples of Maass forms and allows us to circumvent these problems.

\subsection{Work of Zwegers and related notation}\label{ZwegersWorkSection}
In this section we summarize the important recent work of Zwegers \cite{ZwegersMockMaass}, which allows us to study the relation between indefinite theta functions and Maass forms. Effectively, Zwegers showed for a large class of indefinite theta functions how to define eigenfunctions of $\Delta$, along with special completion terms which correct their (non)-modularity. What is especially useful in our case, is the theory which Zwegers provides for describing when these completion terms vanish. To describe the setup, suppose that $A$ is a symmetric $2\times 2$ matrix with integral coefficients such that the quadratic form $Q$ defined by $Q({r}) := \frac12 {r}^T A {r}$
is indefinite of signature $(1,1)$, where ${r}^T$ denotes the transpose of ${r}$.  Let $B({r},\mu)$ be the associated bilinear form given by
\begin{equation}\notag 
 B({r},\mu) := {r}^T A \mu = Q({r}+\mu)-Q(r)-Q(\mu),
\end{equation}
and  take vectors $c_1,c_2\in \R^2$ with $Q(c_j)=-1$ and $B(c_1,c_2)<0$.  In other words, we are assuming that $c_1$ and $c_2$ belong to the same one of the two components of the space of vectors $c$ satisfying $Q(c)=-1$.  We denote this choice of component by $$C_Q:=\{x \in \R^2\mid Q(x)=-1,B(x,c_1)<0\}.$$  It is easily seen that $Q$ splits over $\mathbb{R}$ as a product of linear factors $Q(r)=Q_0(Pr)$
for some (non-unique) $P\in \GL_2(\R)$, where $Q_0(r):=r_1r_2$ with $r=(r_1,r_2)$.
Note that $P$ satisfies $A=P^T(\begin{smallmatrix} 0&1\\1&0\end{smallmatrix})P$.
The choice of $P$ is not unique, however we fix a $P$ with sign chosen so that $P^{-1}\binom{ \hspace{2mm}1}{-1}\in C_Q$. Then, for each $c\in C_Q$, there is a unique $t \in \mathbb{R}$ such that
 \begin{equation}\label{eq:c(t)}
 c = c(t):=  P^{-1}\begin{pmatrix}  e^t \\ -e^{-t} \end{pmatrix}.
 \end{equation}
Additionally, for $c\in C_Q$ we let $c^\perp =c^\perp(t):=P^{-1}\colvector{e^t}{e^{-t}}$.  Note that $B(c,c^\perp)=0$, and $Q(c^\perp)=1$.  It is easily seen that these two conditions determine $c^\perp$ up to sign.

Set
\begin{align*}
\rho_A(r):=\rho_{A}^{c_1,c_2}(r):=&\frac{1}{2} \Big(1-\sign \big(B(r,c_1)B(r,c_2)\big) \Big)
,
\end{align*}
and for convenience, let
$\rho_A^{\perp}:=\rho_A^{c_1^{\perp},c_2^{\perp}}$.
Then, for  $c_j=c(t_j)\in C_Q$, Zwegers defined the function
\begin{equation*} \label{Phidef}
\begin{aligned}
 \Phi_{a,b}(\tau)=\Phi_{a,b}^{c_1,c_2}(\tau)  :&= \sign(t_2-t_1) v^{\frac{1}{2}} \sum_{r \in a+\Z^2} \rho_A(r) e( Q(r)u+ B(r,b))K_0(2\pi Q(r)v) \\
 & \quad + \sign(t_2-t_1) v^{\frac{1}{2}} \sum_{r \in a+\Z^2} \rho_A^\perp (r) e( Q(r)u+B(r,b))K_0(-2\pi Q(r)v)
 .
\end{aligned}
\end{equation*}
Note in particular that
\begin{equation*}
\label{plus}
\Phi_{a, b}^+(\tau)=\sign(t_2-t_1) \sum_{r \in a+\Z^2} \rho_A(r) e(B(r,b))q^{Q(r)}.
\end{equation*}
Here we used that in the proof of convergence in \cite{ZwegersMockMaass} it is shown that for the first sum in the definition of $\Phi_{a,b}$, $Q$ is positive definite whereas in the second sum $Q$ is negative definite from.
Given convergence of this series, it is immediate from the differential equation satisfied by $K_0$ that $\Phi_{a,b}$ is an eigenfunction of the Laplace operator $\Delta$ with eigenvalue $1/4$.

Zwegers then found a completion of $\Phi_{a,b}$. Moreover, he gave useful conditions to determine when the extra completion term vanishes.  To describe this, we first consider for $c\in C_Q$ the $q$-series
 \[
 \varphi_{a,b}^c(\tau) := {v}^{\frac{1}{2}}\sum_{{r}\in a+\Z^2} \alpha_{t} \big({r} {v}^{\frac{1}{2}} \big) q^{Q({r})}e(B({r},b))
 ,
 \]
 $t$ as defined in \eqref{eq:c(t)} and
 \[
 \alpha_{t}({r}):=
\begin{cases} \displaystyle{\int_{t}^\infty} e^{-\pi B({r},c(x))^2}dx & \mbox{ if }B({r},c)B\big({r},c^\perp\big)>0, \\[2ex]
 -\displaystyle{\int_{-\infty}^{t}} e^{-\pi B({r},c(x))^2}dx & \mbox{ if }B({r},c)B\big({r},c^\perp \big)<0, \\
 0 & \mbox{ otherwise, }
 \end{cases}
%
 \]
These functions satisfy the following transformation properties.
\begin{lemma}\label{Zlem}
For $c\in C_Q$ and $a,b\in\R^2$, we have that
\begin{align*}
\varphi_{a+\lambda,b+\mu}^c &= e(B(a,\mu))\varphi_{a,b}^c\quad\mbox{for all $\lambda\in \Z^2$ and }\mu\in A^{-1}\Z^2, \\
\varphi_{-a,-b}^c &= \varphi_{a,b}^c, \\
\varphi_{\gamma a,\gamma b}^{\gamma c} &= \varphi_{a,b}^c \quad \mbox{for all }\gamma\in \mathrm{Aut}^+(Q,\Z),
\end{align*}
where
 \[ \mathrm{Aut}^+(Q,\Z):=\big\{\gamma \in \operatorname{GL}_2(\R)\big\vert \gamma\circ Q=Q,\gamma\Z^2=\Z^2,  \gamma(C_Q)=C_Q, \det(\gamma)=1\big\}. \]
\end{lemma}

Zwegers' main result is as follows,
where
\begin{equation}\label{ZwegersPhiHatDefn}
\widehat\Phi_{a,b}(\tau)=\widehat\Phi_{a,b}^{c_1,c_2}(\tau):=v^{\frac12}\sum_{r\in a+\Z^2}q^{Q(r)}e(B(r,b))\int_{t_1}^{t_2}e^{-\pi vB(r,c(x))^2}dx
.
\end{equation}

\begin{theorem}\label{Zthm}
The function $\Phi_{a,b}$ converges absolutely for any choice of parameters $a,b$, and $Q$ such that $Q$ is non-zero on $a+\Z^2$. Moreover, the function $\widehat \Phi_{a,b}$ converges absolutely and can be decomposed as
\begin{equation}\notag 
 \widehat{\Phi}^{c_1,c_2}_{a,b} = \Phi^{c_1,c_2}_{a,b}+\varphi_{a,b}^{c_1}-\varphi_{a,b}^{c_2}
 .
\end{equation}
Moreover, it satisfies the elliptic transformations
 \begin{align*}
\widehat{\Phi}^{c_1,c_2}_{a+\lambda,b+\mu} &= e(B(a,\mu))\widehat{\Phi}^{c_1,c_2}_{a,b}\quad \mbox{for all $\lambda\in \Z^2$ and }\mu\in A^{-1}\Z^2,\\
\widehat{\Phi}^{c_1,c_2}_{-a,-b} &= \widehat{\Phi}^{c_1,c_2}_{a,b},
\end{align*}
and the modular relations
\begin{align*}
  \widehat{\Phi}^{c_1,c_2}_{a,b}(\tau+1) & = e\bigg(-Q(a)-\frac12 B\big(A^{-1}A^*,a\big)\bigg)\widehat{\Phi}^{c_1,c_2}_{a,a+b+\frac12 A^{-1}A^*}(\tau), \\
  \widehat{\Phi}^{c_1,c_2}_{a,b}\bigg(-\frac{1}{\tau}\bigg) & = \frac{e(B(a,b))}{\sqrt{-\det{(A)}}} \sum_{p\in A^{-1}\Z^2\pmod{\Z^2}} \widehat{\Phi}^{c_1,c_2}_{-b+p,a}(\tau),
\end{align*}
where $A^*:=(A_{11},\ldots,A_{rr})^{\mathrm T}$.
\end{theorem}

These results can be conveniently repackaged in the language of Theorem \ref{mainthm2} as follows.  We note that the addition of the transformation results involving $(a^*,b^*)$ is based on the discussion of the proof of (14) in \cite{ZwegersMockMaass}.
For future reference, we also
define the equivalence relation on $\R^2$
\[(a,b)\sim(\alpha,\beta)\] if $a\pm \alpha\in\Z^2$ and $b\pm \beta=:\mu\in\Z^2$ with $B(a,\mu)\in\Z$ (note that the two $\pm$ are required to have the same sign). 

\begin{proposition}\label{MainThm2ZwegersResult}
If $a,b$ are chosen so that $(\gamma a,\gamma b)\sim(a,b)$ for some $\gamma\in \mathrm{Aut}^+(Q,\Z)$ with $\gamma c_1=c_2$, then
\[
\widehat{\Phi}^{c_1,c_2}_{a,b} = \Phi^{c_1,c_2}_{a,b}
.
\]
 In particular, it is a Maass form (with a multiplier) on a congruence subgroup.
\end{proposition}

The connection of the series $S_{a,b;M}$ with Maass forms (once these series are decorated with the proper modified Bessel functions of the second kind) follows from Zwegers' work, given certain special conditions. To describe these, we define an equivalence relation on the set of pairs $(a,b)$. We also set
\begin{equation}\label{gm}
\gamma_M:=\bigg(\begin{matrix}M&M-1\\ M+1&M\end{matrix}\bigg),
\end{equation}
which is useful for our purposes as it lies in $\mathrm{Aut}^+(Q,\Z)$ and satisfies $\gamma_M c=c'$. 

Finally, for a generic vector $x=(x_1,x_2)$, we let
$$
x^*:=(-x_1,x_2).
$$

\subsection{Quantum modular forms and the map $F\mapsto F^+$}\label{MaassFormsHolomorphization}

In this section, we review Lewis and Zagier's construction \cite{LewisZagier2} of period functions for Maass waveforms, and following Zagier \cite{Za1} indicate how so-called quantum modular forms may be formed using them. We also use this construction in the proof of Theorem \ref{mainthm1}, as we shall see that the $q$-hypergeometric forms of the associated quantum modular forms are essential for showing cuspidality of the Maass waveforms.

We begin by recalling the definition of quantum modular forms (see \cite{Za1} for a general survey).

\begin{definition}\label{cocycle}
For any subset $X\subseteq\mathbb P^1(\Q)$, a function $f\colon X\rightarrow\C$ is a {\it quantum modular form} with {\it quantum set $X$} of weight $k\in\frac12\Z$ on a congruence subgroup $\Gamma$  if for all $\gamma\in\Gamma$, the cocycle ($|_k$ the usual slash operator)
\[r_{\gamma}(x):=f|_{k}(1-\gamma)(x)\]

\noindent extends to an open subset of\, $\R$ and is real-analytic.

\end{definition}

\begin{remark}
Zagier left his definition of quantum modular forms more open only requiring for $r_\gamma$ to be ``nice''. In general, one knows one is dealing with a quantum modular form if it has a certain feel, which Zagier brilliantly explained in his several motivating examples.
\end{remark}
The first main  example Zagier gave, and the one most relevant for us here, is that of quantum modular forms attached to the positive (and negative) coefficients of Maass forms. Although Zagier only worked out this example explicitly in one case, and the work of Lewis and Zagier only studied Maass cusp forms of level one, for our purposes it is important to consider a more general situation.  This is described in the following result, which extends observations of Lewis and Zagier for Maass Eisenstein series of Li, Ngo, and Rhoades for special examples in \cite{RobMaass},
and where, for a Maass form $F$ on a congruence subgroup $\Gamma$, we set
\[
\Gamma_F:=\Gamma\cap\big\{\gamma\in\Gamma : F \text{ is cuspidal at } \gamma^{-1}i\infty\big\}
.
\]

\begin{theorem}\label{MaassQMFThm}
Let $F$ be a Maass waveform on a congruence subgroup $\Gamma$ with eigenvalue $1/4$ under $\Delta$ which is cuspidal at $i\infty$. Then $F^+$ defines a quantum modular form of weight one on a subset $X\subseteq\mathbb P^1(\Q)$ on $\Gamma_F$. Moreover, $F$ is cuspidal exactly at those cusps which lie in the maximal such set $X$.
\end{theorem}

\begin{remarks}\noindent
\begin{enumerate}[leftmargin=*,align=left]
\item The quantum modular form defined by $F^+$ may formally be given on all of $\mathbb P^1(\Q)$. This is done by considering asymptotic expansions of $F^+$ near the cusps, instead of simply values. This consideration leads to Zagier's notion of a {\it strong quantum modular form}.
\item There is also a quantum modular form associated to the negative coefficients of $F$, which is also a part of the object corresponding to $F$ under the Lewis-Zagier correspondence of \cite{LewisZagier2}.
\end{enumerate}
\end{remarks}
\begin{proof}[Sketch of proof of Theorem \ref{MaassQMFThm}]
The key idea, already present in \cite{LewisZagier2}, is to realize $F^+$ as an integral transform of $F$ defined in \eqref{pluspart}. To describe this, we require the real-analytic function $R_{\tau}$
given by ($z=x+iy$ with $x,y\in\R$)
\[
R_{\tau}(z):=\frac{y^{\frac12}}{\sqrt{(x-\tau)^2+y^2}}
.
\]
This function is an eigenfunction of $\Delta$ with eigenvalue $1/4$. For
two real-analytic functions $f,g$ defined on $\mathbb H$, we also consider their Green's form
\[[f,g]:=\frac{\partial f}{\partial z}gdz+\frac{\partial g}{\partial \overline{z}}fd\overline{z}.\]

Then Lewis and Zagier showed (see also Proposition 3.5 of \cite{RobMaass} for a direct statement and a detailed proof) that
\[
F^+(\tau)=-\frac 2{\pi}\int_{\tau}^{i\infty}\big[F(z),R_{\tau}(z)\big]
.
\]
This formula, which may also be thought of as an Abel transform, can also be rephrased as in the proposition of Chapter II, Section 2 of \cite{LewisZagier2} in the
following convenient form:
\begin{equation}\label{FPlusAltInt}
F^+(\tau)\
=
\mathcal C\int_{\tau}^{i\infty}\Bigg(
\frac{\partial F(z)}{\partial z}\frac{y^{\frac12}}{(z-\tau)^{\frac12}(\overline z-\tau)^{\frac12}}dz
+\frac i4 F(z)\frac{(z-\tau)^{\frac12}}{y^{\frac12}(\overline z-\tau)^{\frac32}}d\overline{z}
\Bigg)
,
\end{equation}
where  $\mathcal C$ is a constant. Now for general functions $f,g$ which are eigenfunctions of $\Delta$ with eigenvalue $1/4$, the quantity $[f,g]$ is actually a closed one-form.  This fact, combined with the modularity transformations of $F$ and the equivariance property
\[R_{\gamma \tau}(\gamma z)=(c\tau+d)R_{\tau}(z)\]
for $\gamma=\big(\begin{smallmatrix} a & b \\ c & d\end{smallmatrix}\big)\in\operatorname{SL}_2(\R)$ directly shows (as in (14) of \cite{Za1}) that
\begin{equation*}\label{MaassQuantumCocycle}
F^+(\tau)-(c\tau+d)^{-1}F^+(\gamma \tau)=-\int_{\gamma^{-1}i\infty}^{i\infty}\big[F(z),R_{\tau}(z)\big]
\end{equation*}
for all $\gamma\in\Gamma_F$.  This last integral converges since, by assumption, $F$ is cuspidal at $\gamma^{-1}i\infty$.  As the integral on the
right hand side of the last formula is  real-analytic on $\R\setminus\{\gamma^{-1}i\infty\}$, this establishes the first claim, if we note that the values of the quantum modular form, if they converge, are given as the
limits towards rational points from above. That is, the value of the quantum modular form at $\alpha\in\Q$ equals
\begin{equation}\label{LimitEquationFPlus}
F^+(\alpha):=\lim_{t\rightarrow0^+} F^+(\alpha+it).
\end{equation}

We next establish the second claim, which states that \eqref{LimitEquationFPlus} exists precisely for  those $\alpha$ for which $F$ is cuspidal.
By the existence of a Fourier expansion at all cusps in
Lemma \ref{Maass0Fourier}, and using the exponential decay of $K_0(x)$ as $x\to\infty$, we find that, for $t>0$,
\begin{equation}\label{FourierExpansionCuspAsympExp}
F(\alpha+it)\approx\frac{\kappa_1}{|c|\sqrt{t}}-\frac{\kappa_2}{|c|\sqrt{t}}\log\big(c^2t\big)
,
\end{equation}
where $\gamma=\big(\begin{smallmatrix}a & b\\ c& d\end{smallmatrix}\big)\in\operatorname{SL}_2(\Z)$ is chosen such that $\gamma \alpha=i\infty$ and we write $f(t) \approx g(t)$ if $f-g$ decays faster than any
polynomial in $t$, as $t\rightarrow0^+$. We have also used the fact that $c\alpha+d=0$ to note that the imaginary part of $\gamma(\alpha+it)$ is $t/|c\alpha+cit+d|^2=1/(c^2t)$. Our goal is to show that \eqref{LimitEquationFPlus} converges if and only if $\kappa_1=\kappa_2=0$.

For this, we also require an estimate on $\frac{\partial F}{\partial z}(\alpha+it)$. To compute this, we note that
$$
\frac{\partial}{\partial z}[F(z)]_{z=\alpha+it}
=\frac{\partial}{\partial z}\bigg[F(\gamma z)\bigg]_{z=\alpha+it}=\frac{1}{j(\gamma, \alpha+it)^2}F'\big(\gamma(\alpha+it)\big),
$$
where $j(\gamma,z):=(cz+d)$.
Now using Lemma \ref{Maass0Fourier}, we obtain
\[
	F'(z)\approx\frac{\partial}{\partial z}\Big(\kappa_1y^{\frac{1}{2}}+\kappa_2y^{\frac{1}{2}}\log(y)\Big)=\frac{i}{4}\Big((\kappa_1+2\kappa_2)y^{-\frac{1}{2}}+\kappa_2y^{-\frac{1}{2}}\log(y)\Big).
\]
Using that $
j(\gamma, \alpha+it)=cit \text{ and } \mathrm{Im}(\gamma(\alpha+it))=1/(c^2t),
$
we obtain that
\[
	\frac{\partial F}{\partial z}(\alpha+it)\approx
	\frac{-i}{4|c|t^{\frac{3}{2}}}\Big(\kappa_1+2\kappa_2-\kappa_2\log\big(c^2t\big)\Big).
\]


To determine when $\lim_{t\to 0^+}F^+(\alpha+it)$ exists, we need the following to converge:
\[
\int_\alpha^{i\infty}\Bigg(\frac{\partial F(z)}{\partial z}\frac{y^{\frac12}}{(z-\alpha)^{\frac12}(\overline{z}-\alpha)^{\frac12}}dz+\frac{i}{4}F(z)\frac{(z-\alpha)^{\frac12}}{y^{\frac12}(\overline{z}-\alpha)^{\frac32}}d\overline{z}\Bigg).
\]
Making the change of variables $z=\alpha+it$ (note that we need to conjugate the second term)  gives
$$
i\int_0^\infty\Bigg( \frac{\partial}{\partial z}\big[F(z)\big]_{z=\alpha+it}(-it)^{-\frac12}+\frac{i}{4}
\overline{F(\alpha+it)}(it)^{-\frac12}\Bigg).
$$
The top part of this integral, say from $1$ to $\infty$, is always convergent, since we assumed that $F$ is cuspidal at $i\infty$. Towards $0$, the integrand behaves like
\[
-\frac{i}{4|c|t^{\frac32}}\Big(\kappa_1+2\kappa_2-\kappa_2\log\big(c^2t\big)\Big)(-it)^{-\frac12}+\frac{i}{4}\Bigg(\frac1{|c|t^{\frac12}}\Big(\overline{\kappa}_1-\overline{\kappa}_2\log\big(c^2 t\big)\Big)(it)^{-\frac12}\Bigg).
\]
 Comparing alike powers then gives that the integral only converges for $\kappa_1=\kappa_2=0$, i.e., if $F$ is cuspidal.

\end{proof}

\section{Proofs of the main results}\label{ProofsSection}
\subsection{Proof of Theorem \ref{mainthm2}}

For any $M\in\N_{\geq2}$, consider the quadratic form $Q(x,y):=\frac12\big((M+1)x^2-(M-1)y^2\big)$ associated to the symmetric matrix
$
A:=\big(\begin{smallmatrix}
M+1 & 0
\\
0 & 1-M
\end{smallmatrix} \big)
$
and for $\ell\in\{1,2\}$ the vectors
\begin{align*}
c_\ell:=\frac{1}{\sqrt{M^2-1}}\big((-1)^\ell(M-1), M+1\big)^T.
\end{align*}
It is easily checked that $Q(c_\ell)=-1$ and $B(c_1,c_2)=-2M<0$, so that these two vectors lie in the same component $C_Q$.
Choose $a=(a_1,a_2)\in\Q^2$ and $b=(b_1,b_2)\in\Q^2$.
Then, for any vector $r=(n,\nu)^{\mathrm T}$, we find that
\[
B(r,c_1)B(r,c_2)=(M^2-1)(\nu - n)(\nu + n),
\]
and thus
\begin{equation*}
\rho_A(a+r)
=
\frac12\Big(1+\sign\big((a_1 - a_2 -\nu + n)(a_1 + a_2 + \nu + n)\big)\Big)
.
\end{equation*}
Given these choices, we find that the family of indefinite theta functions $S_{a,b;M}$ may be understood in Zwegers' notation via the relation
\[\Phi_{a,b}^+=\operatorname{sgn}(t_2-t_1) e\big((M+1)a_1b_1-(M-1)a_2b_2\big)S_{a,b;M}.\]
Since $\gamma_M$, defined in \eqref{gm}, can easily be verified to lie in $\operatorname{Aut}^+(Q,\Z)$ and  $\gamma_Mc_1=c_2$, the theorem then follows from Proposition \ref{MainThm2ZwegersResult} if  $(\gamma_M a,\gamma_M b)\sim(a,b)$. We next prove the theorem if $(\gamma_M a,\gamma_M b)\sim(a^*,b^*)$ and $(\gamma_M a^*,\gamma_M b^*)\sim(a,b)$. The key step is to show that the involution $(a,b)\mapsto(a^*,b^*)$ fixes $\widehat \Phi_{a,b}^{c,c'}$. For this, we compute a parameterization $c(t)$ of $C_Q$.
We  find that a suitable choice for $P$ is given by
 $
 P=\frac1{\sqrt2}
 \Big(
 \begin{smallmatrix}
\sqrt{M+1} & \sqrt{M-1}
\\
\sqrt{M+1} & -\sqrt{M-1}
 \end{smallmatrix}
\Big)
 .
 $
 Then we obtain that
$
 c(t)
 =
\bigg(\begin{smallmatrix}
 \sqrt{\frac{2}{M+1}}\sinh(t)
 \\
 \sqrt{\frac{2}{M-1}}\cosh(t)
 \end{smallmatrix}\bigg)
 .
$
In this parameterization, we have $t_\ell=(-1)^{\ell+1}\operatorname{arcsinh}(-\sqrt{(M-1)/2})$ for $\ell\in\{1,2\}$, and  $c(-t)=c^*(t)$. Hence, by sending $x\mapsto-x$ and $r\mapsto r^*$ in \eqref{ZwegersPhiHatDefn}, we find that $\widehat \Phi_{a^*,b^*}^{c_1,c_2}=\widehat \Phi_{a,b}^{c_1,c_2}$. We then obtain, using Lemma \ref{Zlem}, that
\begin{equation*}
\begin{aligned}
2\widehat \Phi_{a,b}^{c_1,c_2}
&
=
\widehat \Phi_{a,b}^{c_1,c_2}+\widehat \Phi_{a^*,b^*}^{c_1,c_2}
\\
&
=
\Phi_{a,b}^{c_1,c_2}+\Phi_{a^*,b^*}^{c_1,c_2}+\varphi_{a,b}^{c_1}-\varphi_{a,b}^{c_2}+\varphi_{a^*,b^*}^{c_1}-\varphi_{a^*,b^*}^{c_2}
\\
&
=
\Phi_{a,b}^{c_1,c_2}+\Phi_{a^*,b^*}^{c_1,c_2}+\varphi_{a^*,b^*}^{c_2}-\varphi_{a,b}^{c_2}+\varphi_{a,b}^{c_2}-\varphi_{a^*,b^*}^{c_2}
\\
&
=
\Phi_{a,b}^{c_1,c_2}+\Phi_{a^*,b^*}^{c_1,c_2}
,
\end{aligned}
\end{equation*}
which shows that the completion terms in $\Phi_{a,b}^{c_1,c_2}$ cancel out, as desired. Finally, we note that since $M\in\N_{\ge2}$, $Q(x,y)$ cannot vanish at rational values $x,y$ unless $x=y=0$ since $M-1$ and $M+1$ are coprime and cannot both be squares. As we have supposed that $a\in\Q^2\setminus\{0\}$, it automatically follows that the quadratic form above cannot vanish on $a+\mathbb{Z}^2$, and hence our choice satisfies the convergence requirement in Theorem \ref{Zthm}.

\subsection{Proof of Theorem \ref{mainthm1}}
We begin by showing the connection of the relevant $q$-series to indefinite theta functions.   To do so, we use the Bailey pairs in the following lemma.   These pairs have the rare and important feature that the $\beta_n$ are polynomials.
\begin{lemma} \label{twopairslemma}
Let $k,\ell \in \mathbb{N}$ with $1 \leq \ell \leq k$.    We have the Bailey pair relative to $1$,
\begin{align}
  \alpha_n &= -q^{(k+1)n^2-n}\big(1-q^{2n}\big)\sum_{\nu=-n}^{n-1}(-1)^\nu q^{-\frac{1}{2}(2k+1)\nu^2 - \frac{1}{2}(2k-(2\ell-1))\nu} \label{firstalpha} \\
\intertext{and}
\beta_n &= H_n(k,\ell;1;q) \cdot  \chi_{n \neq 0} \label{firstbeta},
\end{align}
and the Bailey pair relative to $q$,
\begin{align}
  \alpha_n &= \frac{1-q^{2n+1}}{1-q}q^{(k+1)n^2+kn}\sum_{\nu=-n}^{n}(-1)^\nu q^{-\frac{1}{2}(2k+1)\nu^2 - \frac{1}{2}(2k - (2\ell-1))\nu} \label{secondalpha}\\
\intertext{and}
\beta_n &= H_n(k,\ell;0;q).  \label{secondbeta}
\end{align}
\end{lemma}

 \begin{proof}
The Bailey pair relative to $1$ was established in \cite[Section 5]{Hi-Lo1}.    The proof of the Bailey pair relative to $q$ follows by using a similar argument.   We begin by replacing $K$ by $k$ and $\ell$ by $k-\ell$ in part (i) of Theorem 1.1 of \cite{Lo1}.   This gives that $(\alpha_n,\beta_n)$ is a Bailey pair relative to $q$, where $\alpha_n$ is given in \eqref{secondalpha} and $\beta_n$ is the $z=1$ instance of
\begin{equation} \label{zcase}
\beta_n(z) = \sum_{n \geq m_{2k-1} \geq \ldots \geq m_1 \geq 0} \frac{q^{\sum_{\nu=1}^{k-1} (m_{k+\nu}^2+m_{k+\nu}) + \binom{m_k+1}{2} - \sum_{\nu=1}^{k-1} m_\nu m_{\nu+1} - \sum_{\nu=1}^{k-\ell} m_\nu}(-z)^{m_k}}{(q)_{m_{2k}-m_{2k-1}}(q)_{m_{2k-1}-m_{2k-2}}\cdot\ldots\cdot (q)_{m_2-m_1}(q)_{m_1}},
\end{equation}
where $m_{2k} := n$.    To transform the above into \eqref{secondbeta}, we argue as in Sections 3 and 5 of \cite{Hi-Lo1}.    We replace $m_1,\dots,m_{2k-1}$ by the new summation variables $n_1,\dots,n_{k-1}$ and $u_1,\dots,u_k$ as follows:
\begin{equation} \label{replacement}
m_{\nu} \mapsto 
\begin{cases}
u_{k-\nu+1} + \cdots + u_k & \text{for $1 \leq \nu \leq k$}, \\
n_{\nu-k} + u_{\nu-k+1} + \cdots + u_k & \text{for $k+1 \leq \nu \leq 2k-1$}. 
\end{cases}
\end{equation}
With $m_0 = n_0 = 0$ and $n_{k} = n$, the inequalities $m_{i+1} - m_i \geq 0$ in \eqref{zcase} for $0 \leq i \leq k-1$ give $u_i \geq 0$ and the inequalities $m_{k+i+1} - m_{k+i} \geq 0$ for $0 \leq i \leq k-1$ then give $0 \leq u_i \leq n_{i+1}  - n_i$.   Thus after a calculation to determine the image of the summand of \eqref{zcase} under the transformations in \eqref{replacement}, we find that
\begin{equation*}
\beta_n(z) = \sum_{n \geq n_{k-1} \geq \cdots \geq n_1 \geq 0} \prod_{\nu = 1}^k \sum_{u_{\nu} = 0}^{n_{\nu} - n_{\nu-1}} \frac{\big(-zq^{\min\{\nu,\ell\} + 2\sum_{\mu = 1}^{\nu-1}n_{\mu}}\big)^{u_{\nu}}q^{\binom{u_{\nu}}{2} + 2\binom{n_{\nu-1} + 1}{2}}}{(q)_{n_{\nu} - n_{\nu-1}}} \begin{bmatrix} n_{\nu} - n_{\nu-1} \\ u_{\nu} \end{bmatrix}.
\end{equation*}
By the $q$-binomial theorem 
\begin{equation} \label{qbin}
\sum_{u=0}^n (-z)^uq^{\binom{u}{2}}\begin{bmatrix} n \\ u \end{bmatrix} = (z)_n,
\end{equation}
each of the sums over $u_{\nu}$ may be carried out, giving
\begin{equation*}
\beta_n(z) = \sum_{n \geq n_{k-1} \geq \cdots \geq n_1 \geq 0} \prod_{\nu = 1}^k q^{2\binom{n_{\nu-1} + 1}{2}} \frac{\big(zq^{\min\{\nu,\ell\} + 2\sum_{\mu = 1}^{\nu-1} n_{\mu}}\big)_{n_{\nu} - n_{\nu-1}}}{(q)_{n_{\nu} - n_{\nu-1}}}.
\end{equation*}
Using the fact that
$$
\begin{bmatrix} n \\ k \end{bmatrix}_q = \frac{(q^{k+1})_{n-k}}{(q)_{n-k}},
$$
we then have 
\begin{align*}
\beta_n(1) &= \sum_{n \geq n_{k-1} \geq \cdots \geq n_1 \geq 0} \prod_{\nu = 1}^k q^{2\binom{n_{\nu-1} + 1}{2}} \begin{bmatrix}  \min\{\nu,\ell\} - 1 + n_{\nu} - n_{\nu-1} + 2\sum_{\mu=1}^{\nu-1} n_{\mu} \\ n_{\nu} - n_{\nu-1} \end{bmatrix} \\
&=\sum_{n \geq n_{k-1} \geq \cdots \geq n_1 \geq 0} \prod_{\nu = 1}^{k-1} q^{2\binom{n_{\nu} + 1}{2}} \begin{bmatrix}  \min\{\nu,\ell - 1\} + n_{\nu+1} - n_{\nu} + 2\sum_{\mu=1}^{\nu} n_{\mu} \\ n_{\nu+1} - n_{\nu} \end{bmatrix},
\end{align*}
in agreement with the $H_n(k,\ell;0;q)$, defined in \eqref{Hdef}.
\end{proof}

\begin{remark}
 The referee has observed that Lemma \ref{twopairslemma} could also be proved by using \eqref{A-Grelationbis} together with ideas from \cite{Wa2}.
 \end{remark}

 With these Bailey pairs we prove the following key proposition.
\begin{proposition}\label{IndefThetaFj}
We have
\begin{align}
F_1(k,\ell;q) & = \sum_{n \geq 0} \sum_{|\nu | \leq n}  (-1)^{n+\nu }q^{(k+1)n^2+kn+\binom{n+1}{2} - \frac12\big((2k+1)\nu ^2 + (2k - (2\ell-1))\nu \big)}\big(1-q^{2n+1}\big), \label{F1identity} \\
F_2(k,\ell;q) & = \frac{1}{2}\sum_{n \geq 0} \sum_{|\nu | \leq n}  (-1)^{n+\nu }q^{(k+1)n^2+kn - \frac12\big((2k+1)\nu ^2 + (2k - (2\ell-1))\nu \big)}\big(1-q^{2n+1}\big), \notag \\ 
F_3(k,\ell;q) & = -\sum_{n \geq 1} \sum_{\nu  = -n}^{n-1}  (-1)^{n+\nu }q^{(k+1)n^2+  \binom{n}{2} - \frac12\big((2k+1)\nu ^2 + (2k - (2\ell-1)) \nu \big)}\big(1+q^n\big), \notag \\ 
F_4(k,\ell;q) & = -2\sum_{n \geq 1} \sum_{\nu = -n}^{n-1}  (-1)^{n+\nu }q^{(k+1)n^2 - \frac12\big((2k+1)\nu ^2 + (2k - (2\ell-1)) \nu \big)}. \notag
\end{align}
\end{proposition}

\begin{proof}
The first two identities follow upon using the Bailey pair in \eqref{firstalpha} and \eqref{firstbeta} in equations \eqref{Baileya=1eq1} and \eqref{Baileya=1eq2}, while the second two use \eqref{secondalpha} and \eqref{secondbeta} in equations \eqref{Baileya=qeq1} and \eqref{Baileya=qeq2}.
\end{proof}

We are now ready to prove our main result.
\begin{proof}[Proof of Theorem \ref{mainthm1}]
We first apply Theorem \ref{mainthm2} to the indefinite theta function representations of the $F_\nu $, given in Proposition \ref{IndefThetaFj}.   We begin with $F_1$.   Using the term $(1-q^{2n+1})$ to split the right-hand side into two sums and then replacing $n$ by $-n-1$ in the second sum, we obtain
\begin{equation*}
F_1(k,\ell;q) = \Bigg(\sum_{n\pm \nu  \geq 0} + \sum_{n\pm\nu  < 0 }\Bigg) (-1)^{n+\nu }q^{(k+1)n^2+kn+\binom{n+1}{2} - \frac12\big((2k+1)\nu ^2 + (2k - (2\ell-1))\nu \big)}.
\end{equation*}
By completing the square, we directly compute that
\begin{equation*}
q^{\frac{(2k+1)^2}{8(2k+3)} - \frac{(2k-2\ell+1)^2}{8(2k+1)}}F_1(k,\ell;q) = \Bigg(\sum_{n\pm\nu  \geq 0 } + \sum_{n\pm\nu  < 0 }\Bigg) (-1)^{n+\nu } q^{\frac{1}{2}(2k+3)\big(n+\frac{2k+1}{2(2k+3)}\big)^2 - \frac{1}{2}(2k+1)\big(\nu +\frac{2k-2\ell+1}{2(2k+1)}\big)^2}.
\end{equation*}
We claim that the right-hand side is equal to $S_{a,b;M}$ defined in \eqref{Sdef}
with $ M=2k + 2, a=(\frac{2k+1}{2(2k+3)},\frac{2k-2\ell+1}{2(2k+1)})^{\mathrm T}$, and $b=(\frac1{2(2k+3)},\frac1{2(2k+1)})^{\mathrm T}.$    The summand is directly seen to match that of \eqref{Sdef}.    To show that the summation bounds are correct, we use the restrictions on $k$ and $\ell$ to verify the inequalities $0<a_1\pm a_2<1$.   For example, to see this for $a_1-a_2$, we note that
\[
a_1-a_2=\frac{{(2k+3)\ell-2k-1}}{(2k+3)(2k+1)}
\]
is positive exactly if $\ell>(2k+1)/(2k+3)$. As $0<(2k+1)/(2k+3)<1$ and $\ell\geq1$, this inequality automatically holds. To check the upper bound, note that $a_1-a_2<1$ exactly if $\ell<\frac{2(k+2)(2k+1)}{2k+3}$. This last expression is always bigger than $k$, and $\ell$ is, by assumption, bounded by $k$, so this inequality holds. The inequalities on $a_1+a_2$ may be checked in a similar manner.
We then show that
\begin{equation*}
\gamma_M a+(\ell-2k-1)\begin{pmatrix}1\\ 1\end{pmatrix}=a^*,\quad \gamma_M a^*+\ell\begin{pmatrix}1\\ 1\end{pmatrix}=a, \quad \gamma_M b-\begin{pmatrix}1\\ 1\end{pmatrix}=b^*,\quad\text{and}\quad
\gamma_M b^*=b
\end{equation*}
and also that $B(a,(-1,-1)^{\mathrm T})=-\ell\in\Z$ Theorem \ref{mainthm2} yields the first claim in Theorem \ref{mainthm1} for $F_1$, namely that it is the generating function for the positive coefficients of a Maass waveform.  We return to the question of cuspidality of this Maass form below, after indicating the related calculations which must be performed on the other $F_j$.

In the case of $F_2$, we find in the same manner that
$F_2(k,\ell;q^2)$ is equal (up to a rational power of $q$) to  $\frac12S_{a,b;M}$,
where $M=4k+3, a=(\frac{k}{2(k+1)},\frac{2k-2\ell+1}{2(2k+1)})^{\mathrm T},\text{ and }b=(\frac1{8(k+1)},\frac1{4(2k+1)})^{\mathrm T}.$
As above, we check that
\begin{equation*}
\gamma_M a+(2\ell-4k-1)\begin{pmatrix}1\\ 1\end{pmatrix}=a^*, \quad\gamma_M a^*+(2\ell-1)\begin{pmatrix}1\\ 1\end{pmatrix}=a, \quad\gamma_M b-\begin{pmatrix}1\\ 1\end{pmatrix}=b^*, \quad\text{and}\quad\gamma_M b^*=b.
\end{equation*}
Here, we also have
$0<a_1\pm a_2<1$, and we compute $B(a,(-1,-1)^{\mathrm T})=-2\ell+1\in\Z$, which establishes the theorem for $F_2$.

 For $F_3$, we use the specializations $M=2k+2, a=(-\frac1{2(2k+3)},\frac{2k-2\ell+1}{2(2k+1)})^\mathrm{T}, b=(\frac1{2(2k+3)}, \frac1{2(2k+1)})^\mathrm{T},$
and find that  $0<a_1+ a_2<1, -1<a_1- a_2<0$
\[
\gamma_M a+(\ell-k)\begin{pmatrix} 1\\ 1\end{pmatrix}=a^*, \quad \gamma_M a^*+(\ell-k-1)\begin{pmatrix} 1\\ 1\end{pmatrix}=a,
\quad\gamma_M b-\begin{pmatrix}1\\ 1\end{pmatrix}=b^*,  \quad\gamma_M b^*=b,
\]
and $B(a,(-1,-1)^{\mathrm T})=k-\ell+1$.

Finally, for $F_4$, we have $M=4k+3, a=(0, \frac{2k-2\ell+1}{2(2k+1)})^\mathrm{T}, b=(\frac1{8(k+1)}, \frac1{4(2k+1)})^\mathrm{T},$
and calculate that $0<a_1+ a_2<1, -1<a_1- a_2<0$, while
\[
\gamma_M a+(2\ell-2k-1)\begin{pmatrix} 1\\ 1\end{pmatrix}=a,
\quad\gamma_M b^*=b, \quad\gamma_M b-\begin{pmatrix}1\\ 1\end{pmatrix}=b^*
,
\]
and $B(a,(-1,-1)^{\mathrm T})=2k-2\ell+1$.

Thus, we have shown that the $q$-series in Theorem \ref{mainthm1} are indeed the positive parts of Maass forms. By the construction of Zwegers' Maass forms via the Fourier expansions in \eqref{Phidef}, we see that the Maass forms here are all cuspidal at $i\infty$ whenever they converge.
Theorems \ref{MaassQMFThm} and \ref{mainthm3} then imply that in fact each of the Maass forms in Theorem \ref{mainthm1} are indeed cusp forms.
\end{proof}

\subsection{Proof of Theorem \ref{mainthm3} }

Theorem \ref{mainthm3} follows directly from Theorem \ref{mainthm1} and Theorem \ref{MaassQMFThm}, together with the observation that the $q$-series in \eqref{FFnsDefn} converge (as they are finite sums) at all roots of unity, which implies that their radial limits exist and equal these values by Abel's theorem. Although Theorem \ref{MaassQMFThm} is only stated for Maass forms with trivial multiplier for simplicity, a review of the proof shows that the method applies equally well to our Maass waveforms with multipliers.

\section{Further questions and outlook}\label{QuestionsSection}

There are several outstanding questions which naturally arise from the main results considered here. In what follows, we outline five interesting directions for future investigation.

{\bf 1).}  As the example of $\sigma,\sigma^*$ indicates, it is worthwhile to look at the negative coefficients of the related Maass form.  Thus, it is natural to ask: are there nice hypergeometric representations for the $q$-series formed by the negative coefficients of the Maass forms $G_{j ,k,\ell}$ in Theorem \ref{mainthm2}? For example, one such series has the shape
\begin{equation*}
\sum_{n, \nu \in \Z \atop |(M+1)n+M-1|<2|(M-1)\nu+M-1-2 \ell|}(-1)^{n+\nu}q^{-\frac{1}{8(M+1)(M-1)}\big((M-1)(2(M+1)n+M-1)^2-(M+1)(2(M-1)\nu+M-1-2\ell)^2\big)}.
\end{equation*}
If so, do they have relations to the $q$-hypergeometric series defining the $F_j$-functions, as $\sigma$ and $\sigma^*$ satisfy?  Such a connection could help explain relationships between passing from positive to negative coefficients of Maass waveforms and letting $q\mapsto q^{-1}$ in $q$-hypergeometric series. Examples of such relationships were observed by Li, Ngo, and Rhoades \cite{RobMaass}, and further commented on in \cite{KRW}. However, the authors were unable to identify suitable Bailey pairs to make this idea work in our case.

{\bf 2).} As in Theorem \ref{Zthm}, we may also think of the Maass forms corresponding to the  $F_j$-functions as components of vector-valued Maass waveforms (as discussed in detail for the $\sigma,\sigma^*$ case in \cite{ZwegersMockMaass}).  Is it possible to find nice $q$-hypergeometric interpretations for the corresponding positive (or negative) coefficients of the other components of such vectors as well? That is, are the $q$-series associated to the expansions of the Maass waveforms at other cusps than $i\infty$ interesting from a $q$-series or combinatorial point of view?

{\bf 3).}
Define the $q$-series
\begin{equation} \notag 
\mathcal{U}_k^{(\ell)}(x;q) := \sum_{n \geq 0}q^{n} (-x)_{n}\bigg(\frac{-q}{x}\bigg)_{n}H_{n}(k,\ell;0;q).
\end{equation}
These are analogous to the series $U_k^{(\ell)}(x;q)$, defined by Hikami and the second author \cite{Hi-Lo1} by
\begin{equation} \notag 
U_k^{(\ell)}(x;q) := q^{-k}\sum_{n \geq 1} q^{n}(-xq)_{n-1}\bigg(\frac{-q}{x}\bigg)_{n-1} H_{n}(k,\ell;1;q).
\end{equation}
At roots of unity the functions $U_k^{(\ell)}(-1;q)$ are (vector-valued) quantum modular forms which are ``dual" to the generalized Kontsevich-Zagier functions
\begin{equation}
  \notag 
  F_k^{(\ell)}(q)
  :=
  q^k
  \sum_{n_1, \dots, n_k\geq 0}
  (q)_{n_k} \,
  q^{n_1^{2} + \cdots + n_{k-1}^{2} + n_{\ell} + \cdots + n_{k-1}} \,
  \prod_{j=1}^{k-1}
  \begin{bmatrix}
    n_{j+1} + \delta_{j,\ell-1} \\
    n_j
  \end{bmatrix},
\end{equation}
in the sense that
\begin{equation}\notag %
F_k^{(\ell)}(\zeta_N) = U_k^{(\ell)}\big(-1;\zeta_N^{-1}\big),
\end{equation}
where $\zeta_N:=e^{2\pi i /N}$.
Are the $\mathcal{U}_k^{(\ell)}(-1;q)$ also quantum modular forms like the $U_k^{(\ell)}(-1;q)$?  Are they related at roots of unity to some sort of Kontsevich-Zagier type series?

{\bf 4).} Using Bailey pair methods one can show that
\begin{align}
\mathcal{U}_k^{(\ell)}(-x;q)
&= \frac{(x)_{\infty} \big(\frac{q}{x}\big)_\infty}{
      (q)_\infty ^2} \notag 
    \\
    &\times
    \vast(
      \sum_{\substack{r,s,t \geq 0 \\ r \equiv s \pmod{2}}} +
      \sum_{\substack{r,s,t < 0 \\ r \equiv s \pmod{2}}}
    \vast)
    (-1)^{\frac{r-s}{2}}x^t
      q^{\frac{r^2}{8}+
        \frac{4k+3}{4} r s + \frac{s^2}{8}+\frac{4k+3-2\ell}{4} r
        +
        \frac{1+2\ell}{4} s + t\frac{r+s}{2}}  . \nonumber
    \end{align}
This is analogous to \cite{Hi-Lo1}
\begin{align}
U_k^{(\ell)}(-x;q)
&=
    -q^{-\frac{k}{2}-\frac{\ell}{2}+\frac{3}{8}}
    \frac{(x q)_{\infty}  \big(\frac{q}{x}\big)_\infty}{
      (q)_\infty ^2} \notag 
    \\
    &\times
    \vast(
      \sum_{\substack{r,s,t \geq 0 \\ r \not \equiv s \pmod{2}}} +
      \sum_{\substack{r,s,t < 0 \\ r \not \equiv s \pmod{2}}}
    \vast)
    (-1)^{\frac{r-s-1}{2}}x^t
      q^{\frac{r^2}{8}+
        \frac{4k+3}{4} r s + \frac{s^2}{8}+\frac{1+\ell+k}{2} r
        +
        \frac{1-\ell+k}{2} s + t\frac{r+s+1}{2}}  .  \nonumber
    \end{align}
What sort of modular behavior is implied by these expansions?

{\bf 5).} As per the discussion in \cite{RobMaass}, there is hope that the Maass forms in Theorem \ref{mainthm1} are related to Hecke characters or multiplicative $q$-series. In fact, such connections were related to all related examples of $q$-hypergeometric examples found in the literature, although finding a general formulation seems intractable at the moment since as the discriminants of the quadratic fields grow, explicitly identifying such characters becomes computationally difficult.

\section*{Acknowledgements}
We are grateful to the referee for many helpful comments, especially the observation that the polynomials $H_n(k,\ell;b;q)$ can be related to the Andrews-Gordon identities and for simplifying the proof of Lemma \ref{twopairslemma}.

\end{document}